\documentclass[a4paper,onecolumn,unpublished,11pt, allowfontchangeintitle]{quantumarticle}
\pdfoutput=1
\usepackage[numbers,sort&compress]{natbib}
\usepackage[utf8]{inputenc}
\usepackage[T1]{fontenc}
\usepackage[polish,spanish,english]{babel}
\usepackage{graphicx,amsmath,amsfonts,enumerate,amsthm,amssymb,mathtools,enumitem,thmtools,hyperref,mathdots,enumitem,centernot,bm,soul,bbm}
\usepackage[capitalise, noabbrev]{cleveref}
\usepackage{float}
\usepackage{physics}
\usepackage{colortbl}
\usepackage[table]{xcolor}
\usepackage{mathrsfs} 
\usepackage{lipsum}
\usepackage{tikz}
\usepackage[aligntableaux=center]{ytableau}
\ytableausetup{smalltableaux}

\usepackage{hyperref}
\usepackage{comment}
\usepackage{pifont}

\newtheorem{thm}{Theorem}
\newtheorem{lem}{Lemma}
\newtheorem{conj}{Conjecture}

\usepackage[normalem]{ulem}

\begin{document}

\title{Comment and correction for \textit{``On Explicit Construction of Simplex $t$-designs''} by M. S.  Baladram}


\author{Jakub Czartowski}
\email{jakub.czartowski@ntu.edu.sg}
\affiliation{School of Physical and Mathematical Sciences, Nanyang Technological University,
21 Nanyang Link, 637371 Singapore, Republic of Singapore}
\orcid{0000-0003-4062-833X}

\date{May 9, 2025}

\maketitle

\begin{abstract} 
In \cite{baladram2018explicit} a new method of constructing simplex designs based on cyclic group on $n$ elements has been proposed. One of the claims put forward therein is existence of 3-point simplex 3-design in dimension $d = 3$. In this manuscript we present explicit counterarguments and suggest a manner to rectify the existing proofs. By doing this, we show that the results presented in \cite{baladram2018explicit} can be utilised to construct simplex 3-designs scaling as $d^2$, which suggest a general scaling of $d^{t-1}$. Finally, we put forward a notion that encompasses the objects conforming with bounds given in \cite{baladram2018explicit}, which we refer to as symmetry-restricted simplex $t$-designs.
\end{abstract}

This manuscript provides a comment and correction of M. S. Baladram's manuscript ``On Explicit Construction of Simplex $t$-designs''; as such, it should be read in conjunction with \cite{baladram2018explicit}.
	
	\section{Introduction}
	\label{Introduction}

Simplex designs, cubatures, or averaging sets are examples of a broader class of structures known as designs, as introduced in the works of Delsarte, Seymour and Zaslavsky, and Levenshtein~\cite{Delsarte1977, SEYMOUR1984213, levenshtein1998a}. In the simplest terms, a $t$-design refers to an auxiliary measure $\sigma'$ used to approximate a target measure $\sigma$ up to degree-$t$ moments. Since summing over a finite set of points is generally much simpler than integrating over a continuous distribution, the study of designs typically focuses on discrete measures with as few points as possible. This is exemplified in the case of quadratures, where $2t$ (possibly complex) points suffice to approximate a $\mathbb{Z}_2$-symmetric measure $\sigma$ up to the $(2t+1)$-st moment. In~\cite{levenshtein1998a}, Levenshtein provided a general method for deriving lower bounds on the size of $t$-designs for arbitrary metric spaces, using orthogonal polynomials defined with respect to the underlying metric.

However, points within a $d$-point simplex $\Delta^{d-1}$ are not equivalent, making even the derivation of Levenshtein-type lower bounds based on average $\epsilon$-ball volume a challenging task. As a result, researchers often construct sets $X \subset \Delta^{d-1}$ with as few points as possible to establish upper bounds on the minimal size of a $t$-design. Moreover, due to the exponential growth in the number of moments that must be matched, it is useful to formulate necessary conditions that can efficiently verify whether a given set $X$ constitutes a $t$-design. One of the most prominent criteria of this type, applicable to complex projective designs, is the Welch bound~\cite{Welch1974, Datta2012WelchGeo}, which uses the $\abs{X}^2$ entries of a Gram matrix to evaluate design properties at arbitrary degree $t$.

Both the problem of minimal constructions and the formulation of Welch-like necessary conditions for designs in the $d$-point simplex $\Delta^{d-1}$ have been addressed in~\cite{baladram2018explicit}. That work attempted to exploit the symmetry of a set $X$ under cyclic permutations of coordinates to reduce the number of moment conditions, requiring only $t$ equations to be checked. Based on this approach, the existence of a 3-point 3-design in the 3-point simplex $\Delta^2$ was claimed.

In this short communication, we demonstrate that the claims presented in~\cite{baladram2018explicit} are invalid. However, they can be corrected by replacing the cyclic subgroup $\mathcal{C}_d$ with the full symmetric group $\mathcal{S}_d$. While this change increases the typical size of the generated set $X$ from linear in $d$ to factorial in $d$ (i.e., $d!$), it retains the favorable property that only $t$ moment constraints must be verified. We begin by presenting progressively general counterexamples to the claims in~\cite{baladram2018explicit}, pinpointing the exact points of failure. Next, we provide a corrected version of Lemma 3.1, which we then propagate to the other results from that work. In particular, we show that the three points from Corollary 4.2, together with their mirror images, form a 6-point 3-design in $\Delta^2$ consistent with a revised version of Theorem 3.2.

We then introduce a method for generating upper bounds on the size of the smallest simplex $t$-designs on the order of $d^{t-1}$, which we conjecture to reflect the true asymptotic scaling of minimal designs. Finally, we observe that the original Theorem 3.2 may still be repurposed to define a broader class of \emph{symmetry-restricted simplex designs}, in which a set of points accurately approximates all moments forming complete orbits under a fixed permutation subgroup $G \subset \mathcal{S}_d$.

	\section{Preliminaries}
	
	In what follows we will consider $t$-designs on the probability simplex or in the space of pseudoprobabilities. To this end, we define
	\begin{equation}
		\widetilde{\Delta}^{d-1} = \qty{\vb{p} \in\mathbb{R}^d:\sum_{i=1}^d p_i = 1}, \quad
		\Delta^{d-1} = \widetilde{\Delta}^{d-1} \cap \mathbb{R}^d_{\geq 0}
	\end{equation}
	which are the natural definitions for $d$-point pseudoprobability and probability space, respectively. There is a natural flat Lebesgue measure $\sigma$ on $d$-point probability simplex $\Delta^{d-1}$ such that $\sigma(x) = \text{const.}$ and $\sigma(\Delta^{d-1}) = 1$. This translates naturally to the space of pseudoprobabilities by setting $\widetilde{\sigma}(x) = \sigma(x)$ for $x\in\Delta^{d-1}$ and $\widetilde{\sigma}(x) = 0$ otherwise.
	
	In what follows we will restrict the $t$-designs we consider to the ones defined by unweighed sets of points. In this sense, a set of points $X = \qty{\vb{p}}\subset\widetilde{\Delta}^{d-1}$ is called a $t$-design the averages over the set and over the measure $\tilde{\sigma}$ agree for all polynomials $f_t$ of degree at most $t$,
	\begin{equation}
		\frac{1}{\abs{X}} \sum_{\vb{p}} f_t(\vb{p}) \equiv \ev{f_t}_X = \ev{f_t}_{\widetilde{\Delta}^{d-1}} \equiv  \int_{\widetilde{\Delta}^{d-1}} f_t(\vb{p}) \dd{\widetilde{\sigma}}(\vb{p}).
	\end{equation}
	Whenever $X\subset\Delta^{d-1}$, we will refer to such a set as a simplex $t$-design. On the other hand, if $X\not\subset\Delta^{d-1}$ and $X\subset\widetilde{\Delta}^{d-1}$, we will refer to it as a simplex $t$-\textit{pseudo}design to highlight that the points approximating measure $\tilde{\sigma}$ have $\tilde{\sigma}(x) = 0$. Note that every simplex $t$-design is at the same time a simplex $t$-pseudodesign
	
	In order to satisfy this requirement it is enough to consider all monomials of the form
	\begin{equation}
		M\qty((k_i)_{i=1}^d) \equiv M(\vb{k}) = \prod_{i=1}^d x_i^{k_i} 
	\end{equation}
	with $k_i\in \mathbb{N},\,\sum_i k_i \leq t$; for convenience, we define $\vb{k} \equiv (k_i)_{i=1}^d$. Evaluation of averages of such monomials is directly related to generalised Beta function $B$ by the following formula
	\begin{equation}
		\ev{M\qty(\vb{k})}_{\Delta^{d-1}} = (d-1)! \text{B}\qty((k_i+1)_{i=1}^d)\qq{where}
		B\qty(\qty(\alpha_i)_{i=1}^n)\equiv \frac{\prod_{i=1}^n\Gamma(\alpha_i)}{\Gamma\qty(\sum_{i=1}^n \alpha_i)}.
	\end{equation}
	
	\section{Comment on Lemma 3.1}
	
	In the following sections we will argue that Lemma 3.1 presented in~\cite{baladram2018explicit} is not correct, which we will illustrate using the arrangement of 3 points in $\Delta^2$ presented in Corollary 4.2; based on this, we will properly elicit the reason behind  failure of the solution in this particular case. 
	
	\subsection{Evaluation of Corollary 4.2}
	
	We begin with explicit evaluation of the 3-point set $X\subset\Delta^2$ given in Corollary 4.2 of~\cite{baladram2018explicit}. To this end, we recall that it is given by a probability vector $(p_1,p_2,p_3)$ and its cyclic permutations, with the entries defined as roots of a polynomial
	\begin{equation}\label{eq:3-point_non_des}
		60x^3 - 60x^2+15x-1 = 60\qty(x^3 - x^2+\frac{1}{4}x + \frac{1}{60}) = 60(x-p_1)(x-p_2)(x-p_3). 
	\end{equation}
	Inspection of the quadratic coefficient shows that the roots of a polynomial defined in this way indeed form a probability vector as $p_1+p_2+p_3 = 1$. Linear coefficient allows to conclude that 
	\begin{equation}
		\frac{1}{3}\sum_{\pi\in\mathcal{C}_3} p_{\pi(1)} p_{\pi(2)} = \frac{1}{12},
	\end{equation}
	which is enough to show that $X$ is a simplex 2-design. Finally, the constant serves to show that
	\begin{equation}
		\frac{1}{3}\sum_{\pi\in\mathcal{C}_3} p_{\pi(1)} p_{\pi(2)}p_{\pi(3)} = p_1p_2p_3 = \frac{1}{60},
	\end{equation}
	which is true for the average of $p_1p_2p_3$ over a flat Lebesgue measure on $\Delta^2$.
	
	The proper point of failure arises from the fact that $p_1\neq p_2 \neq p_3\neq p_1$. This forces the following sums to disagree,
	\begin{equation}\label{eq:av_inequal}
		\frac{1}{3}\sum_{\pi\in\mathcal{C}_3} p_{\pi(1)} p_{\pi(2)}^2 \neq \frac{1}{3}\sum_{\pi\in\mathcal{C}_3} p_{\pi(1)}^2 p_{\pi(2)} .
	\end{equation}
	Explicit evaluation yields
	\begin{equation}
		\abs{\ev{p_1 p_2^2}_X - \ev{p_1 p_2^2}_{\Delta^2}} = 
		\abs{\ev{p_2 p_1^2}_X - \ev{p_2 p_1^2}_{\Delta^2}} \approx 0.00481
	\end{equation}
	while their average $\frac{1}{2}\qty(\ev{p_1 p_2^2}_X + \ev{p_2 p_1^2}_X)$  agrees with that over the entire simplex, $\ev{p_1 p_2^2}_{\Delta^2}$. 
	
	Let us make this point more abstractly in terms of degree decompositions of $t$ polynomials over simplices $\Delta^{d-1}$.
	
	\subsection{Decomposition of polynomials over probability simplex}
	
	To this end, we recall definition put forward in~\cite{baladram2018explicit} of symmetrisation of monomials with respect to the cyclic group $\mathcal{C}_d$
	\begin{equation}
		F_{\mathcal{C}_d}\qty(\vb{k})
		= \sum_{\pi\in \mathcal{C}_d} M\qty(\pi(\vb{k}))
		= \sum_{\pi\in \mathcal{C}_d} \prod_{i=1}^d x_{\pi(i)}^{k_i} 
	\end{equation}
	with $\pi(\vb{k})$ defined by $\pi(\vb{k})_i = k_{\pi(i)}$; note the change of notation with respect to~\cite{baladram2018explicit}, purpose of which will become clear in later part of the manuscript. Note that we are working with probabilities, thus the vectors are sum-normalised, $\sum_i x_i = 1$. As a consequence, the following property holds
	\begin{equation}
		F_{\mathcal{C}_d}\qty(\vb{k}) = \sum_{j=1}^d F_{\mathcal{C}_d}\qty((k_i + \delta^{ij})_{i=1}^d).
	\end{equation}
	
	To this end, let us consider explicitly the case of $d=3$ with $\sum_i k_i\leq 3$. we have the following system of polynomials, each expressed in terms of $F_{\mathcal{C}_3}(vb{k})$ such that $\sum_i k_i = 3$,
	\begin{subequations}
		\begin{align}
			F_{\mathcal{C}_3}(1,0,0) & = F_{\mathcal{C}_3}(2,0,0) + 2F_{\mathcal{C}_3}(1,1,0) \nonumber \\
            &= F_{\mathcal{C}_3}(3,0,0) + 3\qty[F_{\mathcal{C}_3}(2,1,0) + F_{\mathcal{C}_3}(1,2,0)] + 2 F_{\mathcal{C}_3}(1,1,1) \\
			F_{\mathcal{C}_3}(1,1,0) & = \qty[F_{\mathcal{C}_3}(2,1,0) + F_{\mathcal{C}_3}(1,2,0)] + F_{\mathcal{C}_3}(1,1,1) \\
			F_{\mathcal{C}_3}(2,0,0)& = F_{\mathcal{C}_3}(3,0,0) + \qty[F_{\mathcal{C}_3}(2,1,0) + F_{\mathcal{C}_3}(1,2,0)]
		\end{align}
	\end{subequations}
        and the remaining polynomials already satisfying $\sum_i k_i = 3$. Since the polynomials $F_{\mathcal{C}_3}(2,1,0)$ and $F_{\mathcal{C}_3}(1,2,0)$ never appear separately, we conclude that
	\begin{equation}
		F_{\mathcal{C}_3}(2,1,0) \notin \operatorname{span}(F_{\mathcal{C}_3}(1,0,0),F_{\mathcal{C}_3}(2,0,0),F_{\mathcal{C}_3}(3,0,0)),
	\end{equation}
        thus providing a counterexample for Lemma 3.1; similar cases can be provided straighforwardly for arbitrary $t \geq 3$. More dauntingly, for $d=4, t = 2$ one can procur the following
	\begin{equation}
		F_{\mathcal{C}_4}(1,0,1,0) \notin \operatorname{span}(F_{\mathcal{C}_4}(1,0,0,0),F_{\mathcal{C}_4}(2,0,0,0)),
	\end{equation}
	thus providing counterexamples for $t\geq2,\,d\geq 4$.

    In order to supply additional intuition we provide visualization of the appropriately symmetrised monomials in Fig.~\ref{fig:polynomials}.
    \begin{figure}[t]
        \centering
        \includegraphics[width=\linewidth]{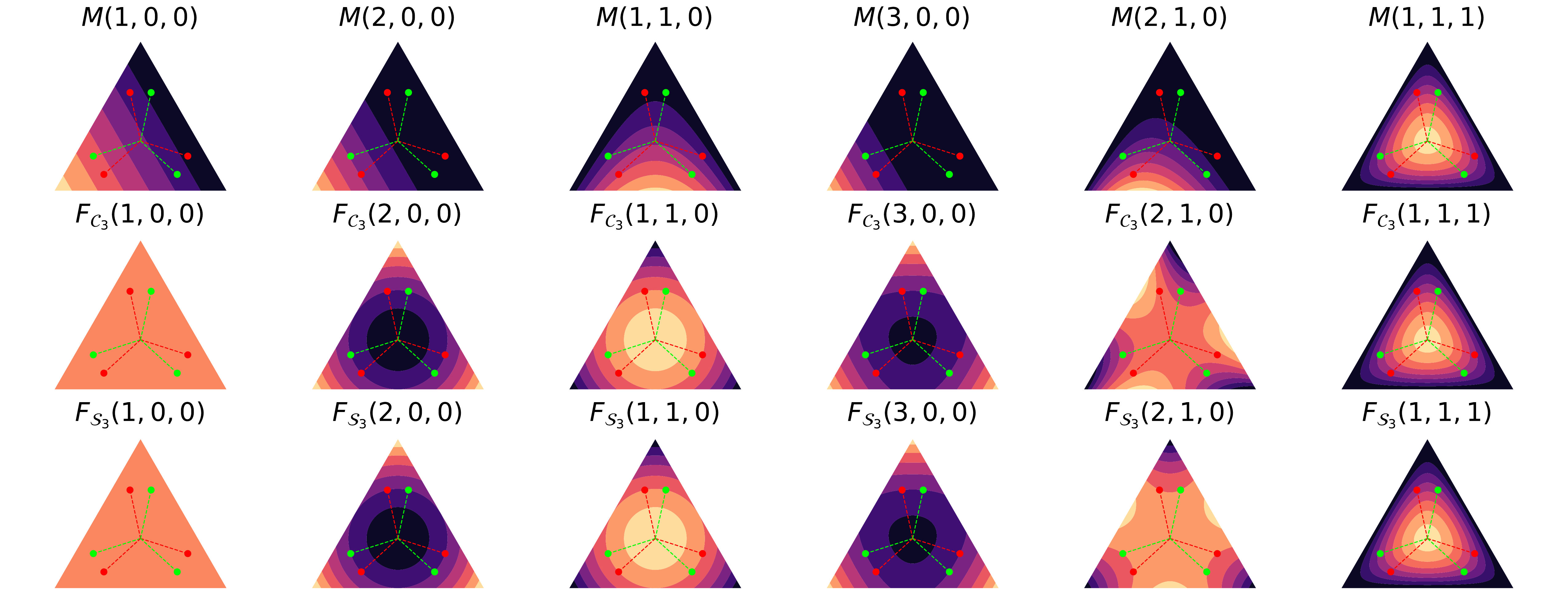}
        \caption{\textbf{Monomials and symmetrised monomials in 3-point simplex}: In the following figures we show contour plots of monomials $M(\vb{k})$ and symmetrised monomials $F_G(\vb{k})$ in 3-points simplex for $\sum k_i \leq 3$ and $G =\mathcal{C}_3,\,\mathcal{S}_3$; Additionally, we show cyclically symmetrised solution of \eqref{eq:3-point_non_des} (red) and its mirror image (green). Note that for all $F_{\mathcal{S}_3}$ all six points lie on the same isoline of the function, thus noting that average over red and green points will be equal. However, in particular case of $F_{\mathcal{C}_3}(2,1,0)$, we find that red points lie on a different isoline than green points, thus leading to different values of averages depending on the mirror symmetry -- equivalent to statement given in Eq. \eqref{eq:av_inequal}; Additionally, $F_{\mathcal{C}_3}(2,1,0)$ is the only polynomial without mirror symmetry, and thus cannot be decomposed of the remaining functions.}
        \label{fig:polynomials}
    \end{figure}
	
	\section{Correction for Theorem 3.2}
	
	In this section we discuss a minor correction which is necessary for Lemma 3.1 to work properly, thus correcting the statement of Theorem 3.2.
	
	\subsection{Discussion of correction}
	
	From now on we will depart from notation introduced in~\cite{baladram2018explicit} and we will define the following polynomials symmetrised with respect to a given subgroup $G\subset \mathcal{S}_d$
	\begin{equation}
		F_G\qty(\vb{k}) = \sum_{\pi\in G} M\qty(\pi(\vb{k}))
	\end{equation}
	and, in particular, we define $F \equiv F_{\mathcal{S}_d}$ as the completely symmetric polynomials with a given set of powers.  Furthermore, note that the following is true
	\begin{equation}\label{eq:coset_decom}
		F\qty(\vb{k}) \equiv F_{\mathcal{S}_d}\qty(\vb{k})  = \sum_{\pi \in\mathcal{S}_d/G} F_{G}\qty(\pi(\vb{k})) =  \sum_{\pi\in \mathcal{S}_d} M\qty(\pi(\vb{k}))
	\end{equation}
	where $\mathcal{S}_d/ G$ is the set of left-cosets of $G$ with respect to $\mathcal{S}_d$.  Additionally it is true that
	\begin{equation}
		F_G\qty(\vb{k}) = F_G\qty(\pi(\vb{k}))
	\end{equation}
	for all $\pi\in G$. 
	
	Note that by replacing the monomials symmetrised with respect to the cyclic group with fully symmetrised monomials we restore validity of Lemma 3.1 and Theorem 3.2 with proofs virtually unaffected, save for replacing the cyclic group $\mathcal{C}_d$ with symmetric group $\mathcal{S}_d$ wherever it appears.
	
	Additionally, let us note that it is enough to consider $\vb{k}$ which are ordered non-increasingly, and thus for a fixed degree $t$ they can be put in one-to-one correspondence with the standard Young tableaux, a notation which we will adopt for simplicity. To provide a basic example, we will write
	\begin{equation}
		F(2,1,0,\hdots,0) = \ydiagram{2,1}.
	\end{equation}
	
	Let us consider a full example of $t = 4$. To this end, we need
	\begin{subequations}
		\begin{alignat}{7}
			\ydiagram{1} & =& \ydiagram{1,1,1,1}& + 3& \ydiagram{2,1,1} &+ 2&\ydiagram{2,2}& + 3&\ydiagram{3,1}& + &\ydiagram{4}, \\
			\ydiagram{1,1} & = &\ydiagram{1,1,1,1}& + 2&\ydiagram{2,1,1} &+& \ydiagram{2,2}& + &\ydiagram{3,1}, \\
			\ydiagram{2} & =&&& \ydiagram{2,1,1} & +& \ydiagram{2,2} &+ 2&\ydiagram{3,1}& + &\ydiagram{4}, \\
			\ydiagram{1,1,1} & =& \ydiagram{1,1,1,1}& +& \ydiagram{2,1,1}, \\
			\ydiagram{2,1} & = &&&\ydiagram{2,1,1}& +& \ydiagram{2,2}& + &\ydiagram{3,1}, \\
			\ydiagram{3} & = &&&&&&&\ydiagram{3,1}& +& \ydiagram{4}, \\
			\ydiagram{1,1,1,1} & = &  \ydiagram{1,1,1,1}.
		\end{alignat}
	\end{subequations}
	Rudimentary check shows that there are only 4 linearly independent equations in this system. Thus, choice of $F(i,0,0)$ with $i \leq 4$ leads to show that Lemma 3.1, modified per description to accommodate full symmetric group instead of cyclic group, works in this case.
	
	The change propagates to corollaries of Theorem 3.2.  However, Corollary 4.1 is essentially unaffected, as probability vectors of the form $(a, b, \hdots, b)$ are invariant under operations from $\mathcal{S}_d/\mathcal{C}_n$, thus creating excess degeneration; similarly, the table under Corollary 4.1 remains valid.
	
	Statement of Corollary 4.2, however, requires rectification, as the actual set of probability vectors constituting the design in question should be given by
	\begin{equation}
		X = \bigcup_{i=1}^s \bigcup_{\pi\in\mathcal{S}_d} \qty{\pi\qty(\vb{p}^{(i)})}
	\end{equation}
	where the change of $\mathcal{C}_d$ to $\mathcal{S}_d$ should be noted. In this respect, $\abs{X} = 6s$ unless $p^{(i)}_2 = p^{(i)}_3$ is set explicitly for some values of $i$, as has been done for the table under Corollary 4.2.
	
	Thus, solution of the equation $60x^3 - 60x^2 + 15x - 1 = 0$, as given in~\cite{baladram2018explicit}, generates a 6-point fully symmetric 3-design in $\Delta^2$. 
	
	\subsection{Lemma 3.1 and Theorem 3.2 -- corrected versions}
	
	Below, we give the corrected versions of the statements given in~\cite{baladram2018explicit}; for this particular section, we use exactly the same notation as in the original work, except for $\mathcal{S}_d$ denoting the symmetric group and minor necessary adjustments.
	
	\begin{lem}[Lemma 3.1 from~\cite{baladram2018explicit} with $\mathcal{S}_d$]
		Let $n, k\in\mathbb{Z}_{>0},\,r_1,\,r_2\in\mathbb{R}$. Define
		\begin{align*}
			V &:= \mathbb{R}[x_1,x_2,\hdots,x_n]\setminus \qty(\sum_{i=1}^n x_i - 1) \\
			F(k_1,k_2,\hdots,k_n) &:= r_1 \sum_{\pi\in\mathcal{S}_d} 
			x^{k_1}_{\pi(1)}
			x^{k_2}_{\pi(2)}
			\hdots
			x^{k_n}_{\pi(n)} - r_2 \text{B}\qty((k_i+1)_{i=1}^n) \in V.
		\end{align*}
		Let $V_k$ be a subspace spanned by $F(j,0,\hdots,0), 0\leq j\leq k$. Then, for non-negative integers $k_1,\hdots,k_n$ such that $\sum_{i=1}^n k_i = k$ one has $F(k_1,\hdots,k_n)\in V_k$.
	\end{lem}
	\begin{thm}[Theorem 3.2 from~\cite{baladram2018explicit} with $\mathcal{S}_d$]\label{thm:new_thm_32}
		Let $n,s,t\in\mathbb{Z}_{>0},\,n\geq2,\,(x_{i,j})_{1,\leq i\leq s,1\leq j\leq n}\in(0,1)^{s\times n}$ where $\sum_{j=1}^n d_{i,j} = 1$ for all $i$. Also, let 
		\begin{equation}
			X = \qty{(x_{i,\pi(1)},x_{i,\pi(2)},\hdots,x_{i,\pi(n)})\in\mathbb{R}^n\mid 1\leq i\leq s,\pi\in\mathcal{S}_d}.
		\end{equation} 
		Then, the multiset $X$ is a simplex $t$-design if and only if
		\begin{equation}
			\frac{1}{sn} \sum_{i=1}^s\sum_{j=1}^n x_{i,j}^k = \frac{\Gamma(n)\Gamma(k+1)}{\Gamma(n+k)}\quad (1\leq k \leq t).
		\end{equation}
	\end{thm}
	\begin{proof}
		Statement of the proofs given in~\cite{baladram2018explicit} need not be changed, save for replacement of cyclic group $\mathcal{C}_n$ with symmetric group $\mathcal{S}_d$.
	\end{proof}
	
	\section{Extension of results}
	
	In the following two subsections we put forward extended results. First subsection presents construction of simplex 3-(pseudo)designs consisting in $d(d-1)$ points. Next subsection discusses an alternative approach to the original Lemma 3.1 and Theorem 3.2, which alters them to accommodate a new notion of symmetry-restricted simplex designs, example of which can be already found in Corollary 4.2 in~\cite{baladram2018explicit}.
	
	\subsection{Extending proposed method of construction}
	
	To this end, we extend the construction proposed in Corollary 4.2 of~\cite{baladram2018explicit} and consider orbits of probability vectors of the form 
	
	$$\qty(\frac{a}{d-2},\hdots,\frac{a}{d-2},b,1-a-b)\in\Delta^{d-1}$$ 
	as candidates for simplex $3$-designs. Such an orbit has exactly $d(d-1)$ elements, thus providing a generic scaling like $d^{t-1}$. Using Theorem 3.2 in its corrected form, we may restrict ourselves to equations of the form
	\begin{equation}
		\frac{a^k}{(d-2)^{k-1}} + b^k + (1-a-b)^k = d\binom{d+k-1}{k}^{-1}.
	\end{equation}
	for $2\leq k \leq t = 3$. Quadratic case $k=2$ can be solved for $b$ as a function of $a$,
	
	\begin{equation}
		b_\pm = \frac{1}{2}\qty(2-a\pm\sqrt{-\frac{d}{d-2} a^2 + 2a - \frac{d-3}{d+1}}).
	\end{equation}
	This determines bounds on $a$,
	\begin{equation}\label{eq:x_restriction}
		\abs{d(a - 1) +2} \leq \sqrt{\frac{2(d-1)(d-2)}{d+1}}
	\end{equation}
	for which $b\notin \mathbb{R}$. Next, substituting the solution into the condition for $k = 3$, it yields a cubic equation
	\begin{equation}
		d(d^2-1) a^3 - 3(d+2)(d^2-1)a^2 + 3(d^2-4)(d-1)a - (d-2)^2(d+1)= 0.
	\end{equation}
	Based on general form of solutions for cubic equations it can be shown that this equation has, generically, three distinct real roots $a_0, a_1, a_2$. Additionally, we find that two out of three satisfy the restriction \eqref{eq:x_restriction} for $d<10$ and only one for $d\geq 10$; moreover, for $d\geq 6$ the solution satisfying \eqref{eq:x_restriction} is a pseudoprobability distribution with $b_- \leq 0$. This allows us to put forward the following theorem.
	
	\begin{thm}
		There exist simplex 3-pseudodesigns with $\abs{X} = d(d-1)$ elements for all $d \geq 2$; for $d \leq 10$ they are simultaneously simplex $3$-designs.
	\end{thm}
	
	\noindent In Tables \ref{tab:proper_solutions} and \ref{tab:improper_solutions} we provide values of $a,\,b$ and $1-a-b$ for dimensions up to $d = 10$.
    
	Based on the above restriction, we can put forward the following conjecture
	\begin{conj}
		Consider space $\mathbb{R}^d/\qty(\sum_{i=1}^d x_i = - 1)$ with measure $\sigma$ such that $\vb{x} \notin\Delta^{d-1} \Rightarrow \sigma(\vb{x}) = 0$, constant over the probability simplex, $\vb{x} \in\Delta^{d-1} \Rightarrow \sigma(\vb{x}) = c$ and normalized so that $\sigma(\Delta^{d-1}) = 1$.
		There exists a $t$-design $X\subset \mathbb{R}^d/\qty(\sum_{i=1}^d x_i = - 1)$ for $\sigma$ with $t \leq d$ with $\abs{X} =  d!/(d-t+1)!$ elements. 
	\end{conj}
	
	\begin{table}[h]
		\centering
		\begin{tabular}{|c|c|c|c|c|c|c|c|c|c|}
			\hline
			$d$ & 3 & $4_1$ & $4_2$ & $5_1$ & $5_2$ & 6 & 7 &  8 & 9\\ \hline\hline
			$a$ &  0.659 & 0.376 & 0.190 & 0.475 & 0.252 & 0.303 & 0.345 & 0.380 & 0.410\\
			\hline
			$b$ &  0.232 & 0.571 & 0.569 & 0.508 & 0.502 & 0.448 & 0.404 & 0.365 & 0.328\\
			\hline
			$1-a-b$ & 0.109 & 0.052 & 0.241 & 0.017 & 0.246 & 0.249 & 0.252 & 0.255 & 0.262 \\
			\hline
		\end{tabular}
		\caption{\textbf{Proper solutions: }Values of $a$, $b$ and $1-a-b$ for $d(d-1)$-point permutationally invariant 3-designs in $\Delta^{d-1}$. Note that for $d = 4, 5$ there are two distinct solutions, denoted with an additional subscript index.}
		\label{tab:proper_solutions}
	\end{table}
	
	\begin{table}[h]
		\centering
		\begin{tabular}{|c|c|c|c|c||c||c||c|}
			\hline
			$d$ & 6 & 7 &  8 & 9 & 16 & 25 & 100  \\ \hline\hline
			$a$ & 0.546 & 0.601 & 0.644 & 0.678 & 0.808 & 0.874 & 0.967\\
			\hline
			$b$ & 0.459  & 0.421 & 0.390 & 0.364 & 0.258 & 0.197 & 0.086\\
			\hline
			$1-a-b$ & -0.006 & -0.022 &  -0.034 & -0.042 & -0.066 & -0.070 & -0.053\\
			\hline
		\end{tabular}
		\caption{\textbf{Improper solutions: }Values of $a$, $b$ and $1-a-b$ for $d(d-1)$-point permutationally invariant 3-designs in $\Delta^{d-1}$.}
		\label{tab:improper_solutions}
	\end{table}
	\subsection{Symmetry restricted simplex designs}
	
	Let us go back to the Eq.~\eqref{eq:coset_decom}, which shows that any fully symmetrised monomial $F(\vb{k})$ can be decomposed into monomials $F_G(\vb{k})$ symmetrised with respect to a subgroup $G$. Thus replacing set $X$ in Theorem \ref{thm:new_thm_32} with an arbitrary set $Y\subset\Delta^{d-1}$ guarantees that 
	\begin{equation}
		\ev{F(\vb{k})}_Y = \ev{F(\vb{k})}_{\Delta^{d-1}} 
	\end{equation}
	but does not say anything about non-symmetric monomials $M(\vb{k})$ by themselves. One may, however, be interested only in certain symmetric averages which are invariant under elements of subgroup $G$. We will call a sequence $\vb{k}$ \emph{$G$-invariant} if
	\begin{equation}
		\forall \pi\in \mathcal{S}_d/G \,\exists \pi' \in G: \pi'(\pi(\vb{k})) = \vb{k} \quad\Longleftrightarrow\quad \forall \pi\in\mathcal{S}_d/G:  F_G(\vb{k}) = F_G(\pi(\vb{k})).
	\end{equation}
	This, in turn, implies that the averages of symmetrised monomials agree,
	\begin{equation}
		F(\vb{k}) = \sum_{\pi\in\mathcal{S}_d/G} F_G(\pi(\vb{k})) = \abs{\mathcal{S}_d/G} F_G(\vb{k}).
	\end{equation}
	
	With this observation in mind we can introduce a notion of \emph{symmetry-restricted simplex designs}. We will call a set $X\subset\Delta^{d-1}$ a $G$-symmetry restricted simplex $t$-designs if the averages over $X$ and over $\Delta^{d-1}$ agree,
	\begin{equation}
		\ev{F_G(\vb{k})}_{\Delta^{d-1}} =
		\ev{F_G(\vb{k})}_{X} 
	\end{equation}
	for all $\vb{k}$ $G$-invariant.
	
	Now, we put forward the following lemma and theorem, which are essentially a second way in which Lemma 3.1 and Theorem 3.2 of~\cite{baladram2018explicit} can be corrected.
	
	\begin{lem}[Lemma 3.1 from~\cite{baladram2018explicit} with $G$-invariant sequences $\vb{k}$]
		Let $n, k\in\mathbb{Z}_{>0},\,r_1,\,r_2\in\mathbb{R}$. Define
		\begin{align*}
			V &:= \mathbb{R}[x_1,x_2,\hdots,x_n]\setminus \qty(\sum_{i=1}^n x_i - 1) \\
			F_G(k_1,k_2,\hdots,k_n) &:= r_1 \sum_{\pi\in G} 
			x^{k_1}_{\pi(1)}
			x^{k_2}_{\pi(2)}
			\hdots
			x^{k_n}_{\pi(n)} - r_2 \text{B}\qty((k_i+1)_{i=1}^n) \in V.
		\end{align*}
		Let $V_k$ be a subspace spanned by $F(j,0,\hdots,0), 0\leq j\leq k$. Then, for non-negative integers $\vb{k} = \qty(k_1,\hdots,k_n)$ such that $\sum_{i=1}^n k_i = k$ and $\vb{k}$ is $G$-invariant one has $F_G(k_1,\hdots,k_n)\in V$.
	\end{lem}
	
	\begin{thm}[Theorem 3.2 from~\cite{baladram2018explicit} with $G$-restricted designs]\label{thm:new_thm_32_2}
		Let $n,s,t\in\mathbb{Z}_{>0},\,n\geq2,\,(x_{i,j})_{1,\leq i\leq s,1\leq j\leq n}\in(0,1)^{s\times n}$ where $\sum_{j=1}^n d_{i,j} = 1$ for all $i$. Also, let 
		\begin{equation}
			X = \qty{(x_{i,\pi(1)},x_{i,\pi(2)},\hdots,x_{i,\pi(n)})\in\mathbb{R}^n\mid 1\leq i\leq s,\pi\in G}.
		\end{equation} 
		Then, the multiset $X$ is a $G$-restricted simplex $t$-design if and only if
		\begin{equation}
			\frac{1}{sn} \sum_{i=1}^s\sum_{j=1}^n x_{i,j}^k = \frac{\Gamma(n)\Gamma(k+1)}{\Gamma(n+k)}\quad (1\leq k \leq t).
		\end{equation}
	\end{thm}
	\begin{proof}
		Both proofs follow in exact same line as Lemma 3.1 and Theorem 3.2 of~\cite{baladram2018explicit}. 
	\end{proof}
	Let us note that in this sense the 3-point configuration in $\Delta^2$ put forward in Corollary 4.2 in~\cite{baladram2018explicit} is a $\mathcal{C}_3$-restricted 3-design, as it reproduces the averages for $\vb{k} = (3,0,0)$ and $(1,1,1)$, which are indeed $\mathcal{C}_3$-invariant, whereas it fails to do it for $(2,1,0)$ and $(1,2,0)$, which are not $\mathcal{C}_3$-invariant.
	
	\section{Summary}
	
%

In this short communication, we have presented a counterexample to Lemma 3.1 and Theorem 3.2 from~\cite{baladram2018explicit}, outlining the mechanism by which the original argument fails. We then proposed two approaches to resolve the identified issue: the first, formalized in Theorem~\ref{thm:new_thm_32}, extends the group from a cyclic group to the full symmetric group; the second introduces a weaker notion of symmetry-restricted simplex $t$-designs, for which Theorem~\ref{thm:new_thm_32_2} can be established.

Symmetry-restricted simplex $t$-designs may, in particular, offer a promising direction for future research. They enable the reconstruction of measures up to degree-$t$ moments, assuming that all relevant polynomial functions exhibit symmetry with respect to a given group—symmetry that may be intrinsic to the problem at hand—thereby reducing the number of evaluation points required. Whether a similar concept can be extended beyond the simplex setting to structures widely used in quantum information theory, such as complex projective designs and unitary designs, remains an open question.

	\section*{Acknowledgments}
	
I would like to thank Dardo Goyeneche and Victor Gonz{\'a}lez Avella for helpful comments during the final stage of preparation of this manuscript.

This work is supported by the start-up grant of the Nanyang Assistant Professorship at the Nanyang Technological University in Singapore, awarded to Nelly Ng.

\bibliographystyle{quantum_abbr}
\bibliography{references}

\begin{thebibliography}{1}

\bibitem{baladram2018explicit}
M.~S. Baladram.
\newblock ``On explicit construction of simplex t-designs''.
\newblock \href{https://dx.doi.org/https://doi.org/10.4036/iis.2018.S.02}{Interdisciplinary information sciences {\bf 24}, 181--184}~(2018).

\bibitem{Delsarte1977}
P.~Delsarte, J.~M. Goethals, and J.~J. Seidel.
\newblock ``Spherical codes and designs''.
\newblock \href{https://dx.doi.org/10.1007/bf03187604}{Geometriae Dedicata {\bf 6}, 363–388}~(1977).

\bibitem{SEYMOUR1984213}
P.~Seymour and T.~Zaslavsky.
\newblock ``Averaging sets: A generalization of mean values and spherical designs''.
\newblock \href{https://dx.doi.org/https://doi.org/10.1016/0001-8708(84)90022-7}{Advances in Mathematics {\bf 52}, 213--240}~(1984).

\bibitem{levenshtein1998a}
V.~Levenshtein.
\newblock ``On designs in compact metric spaces and a universal bound on their size''.
\newblock \href{https://dx.doi.org/S0012-365X(98)00074-0}{Discrete Math. {\bf 192}, 251}~(1998).

\bibitem{Welch1974}
L.~Welch.
\newblock ``Lower bounds on the maximum cross correlation of signals (corresp.)''.
\newblock \href{https://dx.doi.org/10.1109/TIT.1974.1055219}{IEEE Transactions on Information Theory {\bf 20}, 397--399}~(1974).

\bibitem{Datta2012WelchGeo}
S.~Datta, S.~Howard, and D.~Cochran.
\newblock ``Geometry of the {Welch} bounds''.
\newblock \href{https://dx.doi.org/https://doi.org/10.1016/j.laa.2012.05.036}{Lin. Alg. App. {\bf 437}, 2455--2470}~(2012).

\end{thebibliography}
	
\end{document}